\documentclass[11pt]{article}
\usepackage{amsthm} 

\usepackage{mathtools}
\DeclarePairedDelimiter{\norm}{\lVert}{\rVert}

\usepackage{amsmath}
\usepackage{amsfonts}
\usepackage{amssymb}
\usepackage{hyperref}
\usepackage[utf8]{inputenc}
\usepackage[T2A]{fontenc}
\hypersetup{colorlinks,
citecolor=black,
filecolor=black,
linkcolor=black,
urlcolor=black}
\usepackage{graphics,epstopdf}
\usepackage[pdftex]{graphicx}
\usepackage{newlfont}\newlength{\defbaselineskip}
\usepackage[left=1in,right=1in,top=1in,bottom=0.8in,footskip=0.25in]{geometry}
\usepackage{setspace}
\usepackage{subfigure}
\usepackage{enumerate}

\usepackage{authblk}

\usepackage{cancel}

\allowdisplaybreaks
\newtheorem{theorem}{Theorem}
\newtheorem{example}[theorem]{Example}
\newtheorem{lemma}[theorem]{Lemma}
\newtheorem{definition}[theorem]{Definition}

\newcommand{\calN}{\mathcal{N}}
\newcommand{\R}{\mathbb{R}}

\usepackage[most]{tcolorbox}

\begin{document}

\date{\today}
\title{A sufficient condition for the existence of solutions of singular linear-quadratic vector equations.}

\author[1]{Rishikesh Yadav}
\author[1]{Axel Flinth}
\affil[1]{Department of Mathematics and Mathematical Statistics, Umeå University, Sweden}

\maketitle

\begin{abstract}
    We study the existence of solutions for systems of linear-quadratic vector equations with singular linear parts. We derive a sufficient condition for small right hand sides. 
\end{abstract}



\section{Introduction}

Consider an equation of the form
\begin{align}\label{eq:bilinear}
\mathcal{A}u+\mathcal{S}(u,u)=\mathcal{C}.
\end{align}
 Here, $U$ and $V$ are finite-dimensional real vector spaces, $\mathcal{A} : U \to V$ is a linear operator, $\mathcal{S}:U \times U \to V$ is bilinear, and $\mathcal{C}\in V$ is a vector. Such equations naturally arise  as quadratic perturbations of linear operator equations. 
Alternatively, they emerge as Taylor approximations to more general non-linear equations. 
We are interested in the case when the \emph{linear operator} $\mathcal{A}$ is \emph{not invertible} and the \emph{right hand side} $\mathcal{C}$ \emph{is small}.

\paragraph{Motivation} Our motivation to study equations of the form \eqref{eq:bilinear} is the application of so-called \emph{homotopy methods} \cite{allgower1993continuation, avila1974feasibility} to solve non-linear equations of the type
\begin{align*}
    F(u_\lambda) = g(\lambda).
\end{align*}
Here, $F:U \to U$ is a non-linear map, and $g:\mathbb{R}\to U$ is a map depending on a parameter $\lambda\in \mathbb{R}$. The homotopy method starts at an (easy to obtain) solution $u_{\lambda_0}$ for a particular $\lambda_0$, and then follows the trajectory $u_\lambda$ of solutions as $\lambda$ is increased (decreased). It is straightforward to show that as long as $F'(u_\lambda)$ is invertible, that trajectory can be obtained as the solution of the differential equation$
    \tfrac{\mathrm{d}}{\mathrm{d} \lambda} u_\lambda = F'(u_\lambda)^{-1}g'(\lambda)$.
The above approach breaks however when $F'(u_{\lambda_0})$ is singular. To overcome this issue, it is natural to switch to solving an equation
\begin{align*}
    F(u_{\lambda_0}) + F'(u_{\lambda_{0}}) \Delta u_\lambda + \tfrac{1}{2}F''(u_{\lambda_0})[\Delta u_\lambda, \Delta u_\lambda] = g(\lambda_0) + g'(\lambda_0) \Delta_\lambda,
\end{align*}
for small values of $\Delta \lambda$. Since $F(u_{\lambda_0})=g(\lambda_0)$, this turns into an equation of the form \eqref{eq:bilinear} where  $\mathcal{S} = \tfrac{1}{2} F''(u_\lambda)$, $\mathcal{A}=F'(u_{\lambda_0})$ is singular, and $\mathcal{C}= g'(\lambda) \Delta_\lambda$ is small. \newline

Note that when $\mathcal{A}$ is invertible, the well-known implicit function theorem  (see for example \cite{kielhofer2012bifurcation}) immediately shows the existence of a solution to \eqref{eq:bilinear} for small $\mathcal{C}$. When $\mathcal{A}$ is not invertible, however,  there appears to be no unified treatment for equations of the type \eqref{eq:bilinear}. 
The present work aims to partially fill this gap. 
We derive a sufficient condition on the relation between $\mathcal{S}$ and $\mathcal{A}$, that guarantees the existence of the solution for $\mathcal{C}$ small enough. 

\paragraph{Related work}
Meini and Poloni \cite{Meini2011} studied \emph{quadratic vector equations} (QVEs), equations of the form $x=a + b(x, x)$, where $b(u,v)=B(u\otimes v)$ is a bilinear mapping. 
This work was further developed from a computational perspective in \cite{Bini2011}. 
Later, a more general study of QVEs was presented in \cite{Poloni2013}, which solved the problem of a special case for a general form of the equation $Mx=a+b(x,x)$, where $a$ and $x$ are componentwise non-negative vectors,  $b$ is a bilinear map from pairs of non-negative vectors to non-negative vectors, and  $M$ is a \emph{nonsingular $M$-matrix}. Motivated by this work, \cite{Guo2015} estimated the minimal non-negative solution of QVEs,  while Sta\'nczy \cite{Stanczy2013} 
investigated a method for nonlinear equations involving bilinear, coercive, and compact forms.

Closely related to our setting are also systems of \emph{bilinear equations} 
\begin{align*}
    \mathcal{S}(u,v) = \mathcal{C}.
\end{align*}
A specialized framework for solving such systems exactly was presented in \cite{Cohen1997}.  In addition, \emph{least-squares} approaches for computing approximate solutions were studied in \cite{Bai2005}. An abstract solution theory over general fields  $\mathbb{F}$ 
was developed in \cite{Johnson2009}. However, this framework concerns \emph{complete} systems, which necessitate that the number of equations exceed $pq$, where $u\in \mathbb{F}^p$ and $v \in \mathbb{F}^q$. 
In particular, the solvability of BLS over finite fields was investigated in \cite{Vinh2009}, with related recent developments in \cite{Pham2025}. More recently, Huth and Joux  \cite{Huth2024} investigated Multi-Party Computation (MPC) using the subfield bilinear collision problem, in which the authors formulated bilinear systems in polynomial form over a finite field and solved them using Gr\"{o}bner basis algorithms \cite{Faugere2002}.



\section{Main result}\label{sec:main}

 Before presenting our main result, let us reformulate \eqref{eq:bilinear} into a convenient form. Let $\Pi$ be the orthogonal projection onto the range of $\mathcal{A}$. If we define $S = (\mathrm{id} - \Pi) \mathcal{S}$, it is clear that in order for \eqref{eq:bilinear} to have a solution, it is necessary that the projected equation
    \begin{align*}
          S(u,u) &= (\mathrm{id}-\Pi)\mathcal{C}.
    \end{align*}
    has one.  If $u_0$ is such a solution,  by redefining $u = u_0 + \tilde{u}$, equivalently we can rewrite \eqref{eq:bilinear} to
    \begin{align*}
        \mathcal{A}\tilde{u} + \mathcal{S}(u_0,\tilde{u}) + \mathcal{S}(\tilde{u},u_0)  + S(\tilde{u},\tilde{u}) = \mathcal{C}- \mathcal{A}u_0 - S(u_0,u_0). 
    \end{align*}
    Define $A\tilde{u} = \Pi( \mathcal{A}\tilde{u} + \mathcal{S}(u_0,\tilde{u}) + \mathcal{S}(\tilde{u},u_0))$, $\sigma = \Pi \mathcal{S}$, $c= \Pi (\mathcal{C}- \mathcal{A}u_0 - S(u_0,u_0) )$ and $\eta = (\mathrm{id}-\Pi)\mathcal{S}$. Note that $c$ is small when $\mathcal{C}$ is small: More formally,  $\lim_{\mathcal{C}\to 0} u_0=0$, so $\lim_{\mathcal{C}\to 0} c =0 $. By considering the components of the above equation in $\mathrm{ran}(A)$ and its orthogonal complement separately, we see that it decomposes into
    \begin{align*}
        A\tilde{u} + \sigma(\tilde{u},\tilde{u}) = c, \quad 
        \eta(\tilde{u},\tilde{u})= (\mathrm{id}-\Pi)(\mathcal{C}- \mathcal{A}u_0 - S(u_0,u_0)) = (\mathrm{id}-\Pi)(\mathcal{C}- \mathcal{S}(u_0,u_0))=0.
    \end{align*}
    We will therefore here forth treat problems of the following form:

    \begin{tcolorbox}
    For $A$ linear, $\sigma$ and $\eta$ bilinear,  $c\in \mathrm{ran}(A)$ small and  $\sigma(u,u) \in \mathrm{ran} (A)$ for all $u\in U$, solve 
    \begin{align} 
        Au + \sigma(u,u)&=c \label{eq:mainlin1} \\
        \eta(u,u) & = 0. \label{eq:mainlin2}
    \end{align}
    \end{tcolorbox}

Before stating and proving the main result, let us define another shorthand.
\begin{definition}
    For fixed $u \in U$, we define the linear operator 
    $$\mathcal{N}(u): \ker(A) \to \mathrm{ran}(\eta), d \mapsto \eta(u,d) + \eta(d,u).$$
    Here, we used the obvious short-hand $\mathrm{ran}(\eta) = \{\eta(u,v) \, \vert \, u,v \in U\} $.
\end{definition}
Since we will use it later, let us remark that $\calN(u)$ obviously is linearly dependent on $u\in U$.


We are now ready to state our main result.
\begin{theorem}
    \label{th:mainresult} Consider a system of equations of the form \eqref{eq:mainlin1}--\eqref{eq:mainlin2}. Assume that 
    \begin{enumerate}[(i)]
        \item $\eta(d,d) =0$ for all $d \in \ker(A)$,
        \item There exists a vector $w$ with $Aw=c$ so that $\calN(w)$ is invertible.
    \end{enumerate}
    Then, \eqref{eq:mainlin1}--\eqref{eq:mainlin2} has a solution for $c$ \emph{small enough}. That is, for every $c_0$ that fulfills $(ii)$, there exists a $t_0>0$ so that \eqref{eq:mainlin1}-\eqref{eq:mainlin2} has a solution for $c=tc_0$ for $t\in [0,t_0[$. 
\end{theorem}

We postpone the proof to the next section, since it is quite technical. Let us here instead show how the theorem can be applied for a concrete example.
\begin{example}\label{ex:1}
    Let 
    \begin{align*}
        A = \begin{bmatrix}
            A^0 &0
        \end{bmatrix}_{(m+k)\times(m+k)},
    \end{align*}
with $A^0\in \mathbb{R}^{m+k,m}$ of full rank. Then we have
\begin{align*}
    \ker(A) =\big\{(0,y): y\in\mathbb{R}^k \big\}, \quad  \mathrm{ran}(A)  = \mathrm{span}(a^0_i, i \in [m]),
\end{align*}
where $a^0_i$ denotes the columns of $A_0$. For some bilinear forms $\sigma_i : U \times U\to \R$, $i=1,\dots, m$ define $$\sigma(u,u)=\sum_{i} \sigma_i(u,u)b_i\in \R^{m},$$ and for some linear form $\ell: \R^m \to \R$ and bilinear map $\mathcal{B}:\R^m\times \R^m \to \R^k$, the bilinear form $\eta:\R^{m+k}\times \R^{m+k}\to \R^k$ through
\begin{align*}
    \eta(u,v) = \tfrac{1}{2}(\ell(u_1)v_2 + \ell(v_1)u_2) + \mathcal{B}(u_1,v_1),
\end{align*}
where we write $\R^{m+k}\ni u=[u_1,u_2]$ for $u_1\in \R^m$, $u_2\in \R^k$. 

We claim that the resulting system of equations fulfills the conditions of Theorem \ref{th:mainresult} for all $c=A^0\gamma$, $\gamma \in \R^m$ with $\ell(\gamma)\neq 0$. First, it is clear that $\sigma(u,u) \in \mathrm{ran} A$ for all $u$, and likewise that $c\in \mathrm{ran} A$. For $d\in \ker A$, we have $d_1=0$, and consequently,
\begin{align*}
    \eta(d,d) = \ell(0)d_2 + \mathcal{B}(0,0)=0,
\end{align*}
so that (i) is satisfied. As for (ii), first note that for $d=[0,d_2]$ and $u=[u_1,u_2]$, we  have
\begin{align*}
    \calN(u)d = \eta(u,d) +\eta(d,u) = \ell(u_1)d_2 + \ell(0)u_2 + \mathcal{B}(0,u_1)  +\mathcal{B}(u_1,0) = \ell(u_1)d_2.
\end{align*}
Hence, $\calN(u)$ is invertible  exactly when $\ell(u_1)\neq 0$. Since $A^0u=A^0\gamma$ if and only if $u_1=\gamma$, the claim follows.

\end{example}

\section{Proof of the main result}
The main idea of the proof will be to apply the Brouwer fixed-point theorem \cite{Brouwer1911}.  
It states that a continuous map that maps a compact convex subset of a finite dimensional-subset onto itself always has a fixed point. We first need to construct said map, beginning with the following:
\begin{lemma}\label{lem:semiinverse}
    Let $A: U \to V$ be a linear operator and $c\in \operatorname{ran} A$. Then, for any solution $w$ of $Aw=c$, there exists an operator $B:V\to U$ with $Bc=w$ and $ABA=A$, or equivalently  $AB\vert_{\mathrm{ran} (A)}=\mathrm{id}$.
\end{lemma}

\begin{proof}
    We want to construct $B:V \to U$, such that $Bc=w$ and $ABA = A$. Extend $\{c\}$ to a basis $\{c,p_2, \dots, p_k, \dots p_n\}$ of $V$, with $\mathrm{span}(c,p_2, \dots,p_k) = \mathrm{ran}(A)$. Now choose $u_i \in U$ with $Au_i=p_i$, with the specific choice $u_1 =w$, and define
        $B\left(\sum_{i=1}^n \lambda_i p_i\right) = \sum_{i=1}^k \lambda_i u_i$. 
    Now, for $p = \sum_{i=1}^k\mu_i p_i \in \mathrm{ran} A$ arbitrary we can calculate $ABp = A\left(\sum_{i=1}^k\mu_i u_i\right) = \sum_{i=1}^k\mu_i p_i = p$, 
    and particularly $B c = u_1=w$. The claim is proven.
\end{proof}

The proof of Theorem \ref{th:mainresult} is based on the Brouwer fixed-point theorem. Let us begin by showing that we can rewrite \eqref{eq:mainlin1}--\eqref{eq:mainlin2} into a such.
\begin{lemma}
    For $u\in U$ in the set $S = \{u \, \vert \, \calN(u) \text{ invertible. }\}$, define
    \begin{align*}
        \alpha : \mathrm{ran}(\eta) \to \ker A, u \mapsto \calN(u)^{-1}\eta(u,u), \quad \Psi : U \to U, u \mapsto u - \alpha(u).
    \end{align*}
    Now, for a given $c\in \mathrm{ran}(A)$ and solution $w$ to $Aw=c$, let $B$ be as in Lemma \ref{lem:semiinverse}, and define
    \begin{align*}
        \varphi(u) = \Psi(B(c-\sigma(u,u))). 
    \end{align*}
    Then, if
    \begin{align} \label{eq:fixedpoint}
         u = \varphi(u).
    \end{align}
    has a solution,  \eqref{eq:mainlin1}--\eqref{eq:mainlin2} has one also.
\end{lemma}
\begin{proof}
Note in particular that $\alpha$ is well-defined due to Assumption (i) of the main theorem.

Now, let $u$ solve \eqref{eq:fixedpoint}. By the definition of $\alpha$, we then have $v-\Psi(v)= \alpha(v)\in \ker A$ for all $v\in U$. In other words, $A\Psi(v)=Av$. 
Consequently, if $u=\varphi(u)$, we get
    \begin{align*}
        Au = A\varphi(u) = A\Psi(B(c-\sigma(u,u))) = AB(c-\sigma(u,u)) = c-\sigma(u,u).
    \end{align*}
    In the last line, we used the construction of $B$ from Lemma \ref{lem:semiinverse} together with $c-\sigma(u,u)\in \mathrm{ran}(A)$.
    Hence, \eqref{eq:mainlin1} holds. To show that also \eqref{eq:mainlin2} does, note that for all $v\in U$
    \begin{align*}
        \eta(\Psi(v),\Psi(v)) &= \eta(v,v) - \eta(v,\alpha(v))-\eta(\alpha(v),v)+ \eta(\alpha(v),\alpha(v)) \\
        &= \eta(v,v) - \calN(v)\alpha(v) + \eta(\alpha(v),\alpha(v)).
    \end{align*}
    Due to $\alpha(v) \in \mathrm{ker}(A)$, $\eta(\alpha(v),\alpha(v))=0$. Hence
    \begin{align*}
        \eta(v,v) - \calN(v)\alpha(v) + \eta(\alpha(v),\alpha(v)) = \eta(v,v)- \calN(v)\calN(v)^{-1}\eta(v,v)=0.
    \end{align*}
    Applying the above to $v =B(c-\sigma(u,u))$ yields \eqref{eq:mainlin2}.
\end{proof}

The last lemma shows that we can concentrate on showing the existence of a fixed point of \eqref{eq:fixedpoint}. To do that, we need the following technical lemma about the continuity properties of the matrix inverse.
\begin{lemma} \label{lem:inverse}
    Let $\gamma \in U$ be a point where $\calN(\gamma)$ is invertible. Let $\sigma_{\min}$ be the smallest singular value of $\calN(\gamma)$. That is, $\sigma_{\min} = \Vert \calN(\gamma)^{-1}\Vert^{-1}$. Then, for $s\in U$ with $\Vert \calN(s)\Vert\leq \sigma_{\min}$, we have
    \begin{align*}
        \Vert \calN(\gamma)^{-1}-\calN(\gamma-s)^{-1}\Vert \leq \frac{1}{\sigma_{\min}(\sigma_{\min} - \Vert \calN(s)\Vert)}.
    \end{align*}
\end{lemma}
\begin{proof}
    Let $\mathrm{inv}$ be the map that maps a matrix onto its inverse. It is well known that $\mathrm{inv}'(A)B= - A^{-1}BA^{-1}$. The fundamental theorem of calculus implies
    \begin{align*}
        \mathrm{inv}(A+\Delta A) = \mathrm{inv}(A) + \int_0^1 \mathrm{inv}'(A+\theta\Delta A) \Delta A \, \mathrm{d}\theta = A^{-1} - \int_0^1 (A+\theta\Delta A)^{-1} \Delta A (A+\theta\Delta A)^{-1} \, \mathrm{d}\theta.
    \end{align*}
    Hence, with $\varsigma$ the smallest singular value of $A$
    \begin{align*}
        \Vert (A+\Delta A)^{-1} - A^{-1}\Vert &\leq \int_0^1 \frac{\Vert \Delta A \Vert}{\Vert A+ \theta\Delta A\Vert^2} \, \mathrm{d}\theta \leq \int_0^1 \frac{\Vert \Delta A \Vert}{(\varsigma - \theta\Vert\Delta A\Vert)^2} \, \mathrm{d}\theta = \frac{\|\Delta A\|}{\varsigma(\varsigma - \Vert \Delta A\Vert)}.
    \end{align*}
    where the last equation is simple calculus. Now, we note that the linearity of $\calN$ implies that $\calN(\gamma-s) = \calN(\gamma)-\calN(s)$. Hence, applying the above for $A=\calN(\gamma)$ and $\Delta A = \calN(s)$ yields the claim.
\end{proof}

We now have all the tools to prove the main theorem.

\begin{proof}[Proof of Theorem \ref{th:mainresult}] As advertised, our strategy is to apply the Brouwer fixed point theorem to the map $\varphi$ in \eqref{eq:fixedpoint}.  
Let $B$ be constructed as in Lemma \ref{lem:semiinverse}, and introduce the notations
\begin{align*}
    \varrho(u,v) = B\sigma(u,v), \quad s(u) = \varrho(u,u).
\end{align*}
We then have $\varphi(u) = \Psi(w-s(u))$. Let us also note that since $\eta$ and $\varrho$ are bilinear, there exists $M,L\geq 0$ with 
\begin{align}
    \Vert\varrho(u,v)\Vert \leq L\Vert u\Vert \cdot \Vert v\Vert, \quad \Vert\eta(u,v)\Vert \leq M\Vert{u}\Vert \cdot \Vert v\Vert. \label{eq:bounds}
\end{align}
Furthermore, since $\calN(w)$ is invertible, and linearly dependent on $w$, there must exist a $\delta>0$ so that 
\begin{align} \label{eq:lowerbound}
    \sigma_{\min}(\calN(w))\geq \delta \Vert w \Vert.
\end{align}
Without loss of generality, we can assume that $\delta\leq 1$.  Now, for an $\varepsilon>0$ we decide later; define
\begin{align*}
    S_\epsilon = \{ u \, \vert \,  \Vert u - \Psi(w)\Vert \leq \varepsilon \}.
\end{align*}
We claim that if $w$ is small enough and $\varepsilon$ is chosen small enough, $\varphi$ is continuous on $S_\epsilon$, and maps it onto itself.

As an auxiliary result, let us prove that there exists a $K\geq 0$ so that
\begin{align} \label{eq:sbound}
    \| s (u) \| \leq K(\|w\|^2 + \varepsilon^2).
\end{align}
for all $u\in S_\epsilon$. To simplify the notation, let us write $u = \Psi(w)+r$. We now use \eqref{eq:bounds} and \eqref{eq:lowerbound} to estimate
\begin{align*}
    \|s(u)\| &= \|\varrho(u,u)\| \leq L\|u \|^2 = L\|\Psi(w)+r\|^2 = L\|w - \alpha(w)+r\|^2 \\
    &\leq 3L (\|w\|^2 + \|\alpha(w)\|^2+\|r\|^2) \leq 3L(\| w\|^2 + \|\calN(w)^{-1} \|^2 \cdot \|\eta(w,w)\|^2 + \varepsilon^2) \\
    & \leq 3L(\|w\|^2 +\delta^{-2}\|{w}\|^{-2}\cdot M^2 \|w\|^4 + \varepsilon^2) \leq K(\|w\|^2 + \varepsilon^2)
\end{align*}
for a properly defined $K$, i.e. \eqref{eq:sbound}.

Let us now get to the bulk of the proof. In order to show that $\varphi$ is continuous on $S_\epsilon$, let us note that
\begin{align*}
    \varphi(u) = \Psi(w-s(u))= (w-s(u))-\calN(w-s(u))^{-1}\eta(w-s(u),w-s(u)).
\end{align*}
It hence suffices to argue that $\calN(w-s(u))= \calN(w)-\calN(s(u))$ is invertible for all $u \in S_\epsilon$. According to Lemma \ref{lem:inverse}, it suffices to show that $\|{\calN(s(u))}\|\leq\sigma_{\min}(\calN(w))=\delta \|w\|$. But combining \eqref{eq:bounds} with the definition of $\calN$ yields $\| \calN(w)\| \leq 2 M\|w\|$ for any $w$, so that 
\begin{align*}
    \|{\calN(s(u))}\| \leq 2M \|s(u)\| \leq 2KM(\|w\|^2+\varepsilon^2),
\end{align*}
which is smaller than $\delta \|w\|/2$ if 
\begin{align} \label{eq:first_specifications}
    \| w \|\leq \frac{\delta}{8KM}, \quad \varepsilon^2\leq \frac{\delta \|w\|}{8KM}.
\end{align}
which is true for $w$ and $\varepsilon$ small enough.

To show that $\varphi$ maps $S_\epsilon$ onto $S_\epsilon$, we need to show that $\Vert{\varphi(u)-\Psi(w)}\Vert<\varepsilon$. We have
\begin{align*}
    \varphi(u)-\Psi(w) &= \Psi(w-s(u))-\Psi(w) = \left(w-s -\alpha(w-s)\right)- \left(w-\alpha(w)\right) \\
    &= \alpha(w)-\alpha(w-s)-s.
\end{align*}
where we wrote $s=s(u)$ to ease the notation somewhat. A short calculation shows that 
\begin{align*}
    \alpha(w)- \alpha(w-s) = (\calN(w)^{-1}-\calN(w-s)^{-1})\eta(w,w) - \calN(w-s)^{-1}(\eta(w,s)+\eta(s,w) - \eta(s,s)).
\end{align*}
With the help of Lemma \ref{lem:inverse} and the bound $\|\calN(s)\|\leq 2M\|s\|\leq \delta \|w\|/2$, we obtain 
\begin{align*}
    \| \calN(w)^{-1}-\calN(w-s)^{-1}\| &\leq \frac{ 2M\|s\|}{\delta \|w\| (\delta \|w\| - M\|s\|)} \leq \frac{ 4M\|s\|}{\delta^2 \|w\|^2} , \\
    \| \calN(w-s)^{-1}\| &\leq \| \calN(w)^{-1}-\calN(w-s)^{-1}\| + ||\calN(w)^{-1}|| \\
    &\leq \frac{ 4M\|s\|}{\delta^2 \|w\|^2 (\delta \|w\| - 2M\|s\|)} + \frac{1}{\delta \|w\|} = \frac{1}{\delta \|w\| - 2M\|s\|} \leq \frac{ 2}{\delta \|w\|}.
\end{align*}
   Consequently,
    \begin{align*}
        \|\alpha(w)- \alpha(w-s)\| &\leq \frac{ 4M\|s\|\| \eta(w,w)\|}{\delta^2 \|w\|^2 }  + \frac{4\| \eta(w,s)\|  + 2\| \eta(s,s)|\|}{\delta \|w\| }  \\
        &\leq \frac{ 4M^2\|s\|}{\delta^2} + \frac{8M\|w\| \|s\|  + 2M \|s\|^2 }{\delta \|w\|}  \leq 4M^2\delta^{-2}\|s\| + 8M\delta^{-1}\|s\| + M\|s||,
    \end{align*}
    where we in the final step used the bound $2M||s||\leq \delta \|w\|/2$. By using the bound \eqref{eq:sbound}, and $\delta\leq 1$ we see that this smaller than
    \begin{align*}
        R(1+\delta^{-2})(\|w\|^2+\varepsilon^2)
    \end{align*}
    for some constant $R$. It is a simple algebraic exercise to show that the expression above is smaller than $\varepsilon$ as long as it is chosen in the interval
    \begin{align*}
         \bigg] \frac{1}{2R(1+\delta^{-2})}- \sqrt{ \frac{1}{4R^2(1+\delta^{-2})} -\|w\|^2}, \frac{1}{2R(1+\delta^{-2})}+\sqrt{ \frac{1}{4R^2(1+\delta^{-2})}- \|w\|^2}\bigg[.
    \end{align*}
    For small $\|w\|$, the left boundary is approximately equal to 
\begin{align*}
    \frac{\|w\|^2}{2R(1+\delta^{-2})}.
\end{align*}
This shows that there are $\varepsilon$ that both is contained in the above interval and fulfills \eqref{eq:first_specifications}. Hence, there exists an $\varepsilon$ so that $\varphi$ maps $S_\epsilon$ onto $S_\epsilon$, and the claim is proven.
\end{proof}

\section{Numerical Example}

\begin{figure}
    \centering
    \includegraphics[width=0.49\linewidth]{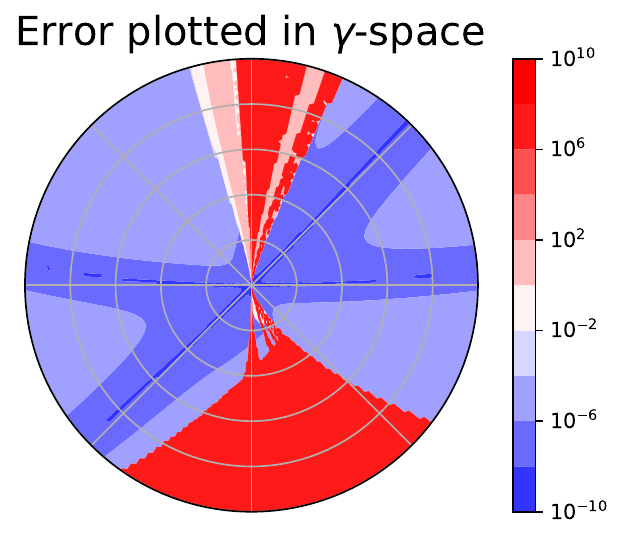} \includegraphics[width=.49\linewidth]{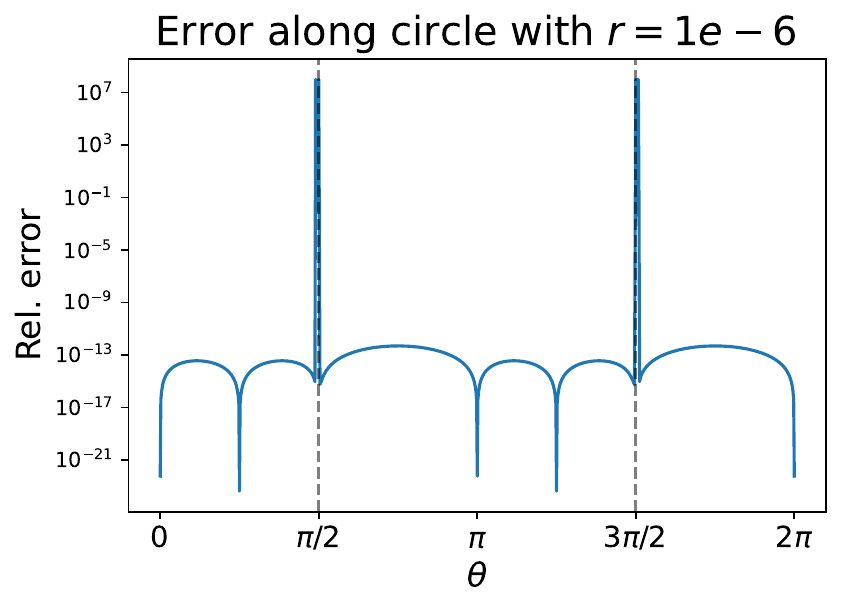}
    \caption{Left: Relative error at iteration $3000$ for different values of $\gamma$ on a circle with radius $5e-2$. Notice that the color depends logarithmically on the error. Best viewed in color. There are regions where no solution is found, but this region is very thinly concentrated around the axis $\gamma[1]=0$ close to the origin. Right: The error in dependence of the angle relative to the $x$-axis for $\gamma$ on a circle of radius $1e-6$. In both plots, the relative error has been capped upwards at $1e8$ and downwards at $1e-11$, to make the figure more legible.}
    \label{fig:errorplots}
\end{figure}

 We perform a small numerical experiment to illustrate Theorem \ref{th:mainresult}, in particular the condition (ii).
 
 \paragraph{Setup} We adopt the setting of Example \ref{ex:1}, with the following concrete choices: $m=k=2$, $A_0\in \R^{4,2}$ a random Gaussian matrix and $\sigma_i(u,u)=\langle{u,W_iu}\rangle$ for random matrices $W_i\in \R^{4,4}$ drawn from the Wishart distribution $W(I,4,8)$ \cite{Eaton2007}. The linear form is chosen as $\ell(u)=\langle{u,e_1}\rangle$, and the bilinear form $\mathcal{B}$ is defined through $\mathcal{B}(u_1,v_1) = \langle{u_1,e_2}\rangle\cdot \langle{v,e_2}\rangle\cdot[1,1]$

\paragraph{Experiment description}The proof of Theorem \ref{th:mainresult} gives us a clear strategy to solve the system \eqref{eq:mainlin1}, \eqref{eq:mainlin2}-- simply apply a fixed-point iteration $u_{k+1}=\varphi(u_k)$, started at the initial solution $u_0=Bc$. We here choose $B$ as the Moore-Penrose inverse. We perform $3000$ such iterations for values of $\gamma \in \R^2$ (remember the notation $c=A_0\gamma$) in a circle of radius $5e-2$ about the origin, and record the relative error
\begin{align*}
    \text{Rel. Error} = \frac{\norm{Au+\sigma(u,u)-c}^2+\norm{\eta(u,u)}^2}{\norm{u}^2}
\end{align*}
at the final iteration. The code used in the experiments will be made available upon request. 

\paragraph{Results} The results are plotted in Figure \ref{fig:errorplots} (left). We see that the fixed point iterations finds a solution for many values of $\gamma$, in particular close to the origin. However, there is a problematic region, which appearantly contains the $y$-axis, which is exactly the set where $\ell(\gamma)=0$ -- i.e., as predicted by our theorem. To make this point clearer, we solve the equation for $\gamma$ along a circle with radius 1e-6. The error in dependence of the angle relative to the $x$-axis is shown in Figure \ref{fig:errorplots} (right). We see that the error explodes occur exactly at $\theta = \pi/2$ and $\theta=3\pi/2$, which is exactly where our theorem predicts it will.
In short, the results are exactly as our theorem predicts.

\section{Conclusion}
In this manuscript, we investigated the existence of a solution to a system of singular linear-quadratic equations. Our main result establishes that under two compatibility conditions of the bilinear and linear terms, a solution exists as soon as the right-hand side is small enough. We illustrated the result with a small numerical experiment. The condition seems to not have been recorded in the literature before. However, it is only a necessary condition, and hence does not completely characterize when a solution exists. More research is hence necessary.

\subsection*{Acknowledgement} Rishikesh Yadav's research was supported by Kempestiftelserna through the project grant ``Iterative algorithms for regularized optimization on measure space'', and both he and Axel Flinth acknowledges support from them. Axel Flinth in addition, acknowledges support from the Wallenberg AI, Autonomous Systems and Software Program (WASP) funded by the Knut and Alice Wallenberg Foundation.



\bibliographystyle{amsplain}
\bibliography{reference}

\end{document}